\documentclass[11pt]{article}
\usepackage{latexsym}
\usepackage{amssymb}
\usepackage{amsmath}
\usepackage[all]{xy}
\usepackage{amsthm}
\usepackage{color}

\usepackage{authblk}

\newtheorem{teo}{Theorem}[section]
\newtheorem{prop}[teo]{Proposition}
\newtheorem{lemma}[teo] {Lemma}
\newtheorem{coro}[teo]{Corollary}

\theoremstyle{definition}
\newtheorem{example}{Example}
\newtheorem{remark}{Remark}
\newtheorem{definition}{Definition}

\def\Hom{\mathop{\mathrm{Hom}}\nolimits}

\def\Ext{\mathop{\mathrm{Ext}}\nolimits}

\def\Ker{\mathop{\mathrm{Ker}}\nolimits}

\def\Kpar{\mathop{\mathrm{K}_{\mathrm{par}}}\nolimits}

\def\ParRep{\mathop{\mathrm{ParRep}}\nolimits}
\def\ParAct{\mathop{\mathrm{ParAct}}\nolimits}
\def\CovPair{\mathop{\mathrm{CovPair}}\nolimits}
\def\Rep{\mathop{\mathrm{Rep}}\nolimits}

\def\Der_par{\mathop{\mathrm{Der}_{\mathrm{par}}}\nolimits}
\def\Int_par{\mathop{\mathrm{Int}_{\mathrm{par}}}\nolimits}

\def \End { \mathop{ \mathrm{End} } \nolimits}

\def\H {\mathop{\mathrm{H}}\nolimits}
\def \id {\mathop{\mathrm{id}} \nolimits}

\def\N{{\mathbb N}}

\def \Z{{\mathbb Z}}

\def\C{{\mathbb C}}
\def\F{{\mathbb F}}

\def\gA {\tilde{A}}

\begin{document}

\title{Cohomology of partial smash products}
\date{}
\author[1]{Edson Ribeiro Alvares}
\author[1]{Marcelo Muniz Alves}
\author[2]{Mar\'\i a Julia Redondo \thanks{The first and second authors acknowledge support by CNPq, Brazil, projects 305882/2013-9 and 457940/2013-1, by AUGM and by the Math Department of UFPR.
The third author is a researcher from CONICET (Argentina) and she acknowledges support by the project PICT 2011-1510, ANPCyT.}}

\affil[1]{Centro Polit\'ecnico, Departamento de Matem\'atica, Universidade Federal do Paran\'a, CP019081, Jardim das Am\'ericas, Curitiba, PR 81531-980, Brazil}
\affil[2]{Instituto de Matem\'atica (INMABB), Departamento de Matem\'atica, Universidad Nacional del Sur (UNS)-CONICET, Bah\'\i a Blanca, Argentina}

\setcounter{Maxaffil}{0}
\renewcommand\Affilfont{\itshape\small}

\maketitle

\begin{abstract}
We define the partial group cohomology as the right derived functor of 
the functor of partial invariants, 
we relate this cohomology with partial derivations and with the partial 
augmentation ideal and we show that there exists a Grothendieck spectral
sequence relating cohomology of partial smash products 
with partial group cohomology and algebra cohomology.
\end{abstract}


\section{Introduction}

The concept of partial group actions and representations was introduced in \cite{E} and \cite{QR},
motivated by the desire to study algebras generated by partial isometries on a Hilbert space $H$. 
More specifically, the initial motivation for introducing partial group actions in \cite{E} 
was to study a certain $\Z$-graded algebra as a smash product with respect to a weaker form of $\Z$-action. 
This construction led to the concept of partial $G$-action on an algebra $A$, which consists of a family of ideals
$\{D_g\}_{g \in G}$ of $A$ and a family of algebra isomorphisms $\alpha_g : D_{g^{-1}} \to D_g$ satisfying some compatibilities. 
The associated partial skew group algebra $A \times_\alpha G$ is the $k$-vector space $\oplus_{g \in G}D_g$ endowed 
with a multiplication that resembles the one that defines a skew group algebra, and coincides with it when $D_g=A$ for every $g$ in $G$.
Partial representations of $G$ appear naturally as an ingredient in the study of the representations of the 
partial skew group algebra $A \times_\alpha G$, see \cite{DE}.

In \cite{DES08} the authors expand the concept of 
partial smash product to that of a partial crossed product, 
with cocycles taking values in multiplier algebras, 
and this approach culminated in a characterization of the $G$-graded algebras which 
are isomorphic to a partial crossed product.
On the other hand, recently it was proved that a large class of $\Z$-graded algebras, the Leavitt path algebras of 
graphs \cite{AAP}, 
can be expressed
as  partial smash products \cite{Goncalves} over the free group generated by the arrows of the underlying quiver.
Among other developments, we may cite also the development of a Galois theory for partial actions \cite{DFP07,BP12,KS14,KS16}. 


Given an action of $G$ on an algebra $B$, every unital ideal of $B$ carries a partial action: if $A$ is such an ideal, with unit $1_A$, then 
a partial $G$-action on $A$ is obtained by defining $D_g$ as the ideal $A \cap g(A)$ and $\alpha_g$ to be the restriction of the map $b \in B \mapsto g(b) \in B$ to the ideal $D_{g^{-1}}$. 
If a partial action arises in this manner, one says that this partial action is \textit{globalizable}, and its globalization is the 
subalgebra $\gA = \oplus_{g \in G} g(A)$. 
It is well-known that if $\gA$ is a (unital) globalization for 
$A$ then the partial smash product $A \times_\alpha G$
and the skew group algebra $\gA [G]$ are Morita equivalent \cite{DE}. 
Therefore, since Hochschild cohomology is a Morita invariant, in principle 
one could substitute $\gA[G]$ for $A \times_\alpha G$ in order to calculate the cohomology of the former. 
However there is a downside to this approach: the globalization $\gA$ may not be a unital algebra, and 
the way that $\gA$ is usually obtained, as the subalgebra generated by vector subspaces of an algebra of functions, 
makes it  hard to describe it explicitly (e.g., by generators and relations). 
Therefore one needs tools to calculate the Hochschild cohomology of $A \times_\alpha G$ that do not involve the globalization $\gA$, and here lies the main contribution of this work.

In Section~2 we recall, to the benefit of the reader, some definitions and fundamental known results
regarding partial actions and partial representations of a group. Here we recall
the definition of partial representation, and show that the category of partial 
representations $\ParRep G$ is equivalent to the category of representations of 
the partial group algebra $\Kpar G$, see \cite{DE}. We also recall the definition 
of partial action of $G$ on an algebra $A$, we recall the construction of the partial smash product 
$A \times_\alpha G$ and we show that the category of representations of the partial smash product 
$A \times_\alpha G$ is equivalent to the category of covariant pairs $\CovPair (A,G)$ whose objects 
are pairs in $\Rep A \times \ParRep G$ with some compatibility property. 
Finally we show that the partial group algebra $\Kpar G$ is in fact a partial smash 
algebra $B \times_\beta G$, see \cite[Thm 6.9]{DE}. 

In Section~3 we define the partial group cohomology as the right derived functor of the functor of partial invariants. As a first step we show that the functor of partial invariants is representable, that is, $(-)^{G_{\mathrm{par}}} \simeq \Hom_{\Kpar G} (B, -)$. Later we relate this cohomology with partial derivations and with the partial augmentation ideal.

In Section~4 we show that there exists a Grothendieck spectral sequence relating cohomology of partial smash products with partial group cohomology and algebra cohomology.

\section{Basic definitions}  

In this section we introduce all the necessary definitions and results that will be used throughout this article. We refer to \cite{DE} for more details.

Let $G$ be a group and $K$ be any field. We denote by $e$ the identity of $G$.

\begin{definition}
A \textit{partial representation} of $G$ on the $K$-vector space $V$ is a map $\pi : G \to \End_K(V )$ such that, for 
all $s, t \in G$, we have:

\begin{itemize}
\item [(a)] $\pi (s) \pi (t) \pi (t^{-1}) = \pi (st) \pi (t^{-1})$; 
\item [(b)] $\pi (s^{-1})\pi (s) \pi (t) = \pi(s^{-1})\pi(st)$;
\item [(c)] $\pi (e) = 1$,
\end{itemize}
where $1 = \id_V$ is the identity map on $V$. 
\end{definition}

In other words, $\pi$ is a partial representation of $G$ if the equality $\pi (s)\pi (t) = \pi(st)$ holds when the two sides are multiplied either by $\pi (s^{-1})$ on the left or by $\pi (t^{-1})$ on the right.

\begin{example}
Every representation of $G$ is a partial representation; moreover, if $H$ is any subgroup of $G$ and $\pi : H \to \End_K(V )$ is a partial representation of $H$, then the map $\tilde {\pi} : G \to \End_K(V )$ given by 
\[ \tilde {\pi }(g) = \begin{cases} \pi (g)  \quad & \mbox{if $g \in H$,} \\ 0  & \mbox{otherwise} \end{cases}\]
defines 
a partial representation of $G$. 
\end{example}

 \begin{example}
Partial representations underlie important algebras generated by partial isometries. Among the most interesting cases are the Cuntz- Krieger algebras \cite{CK}, that is, universal $\C^*$-algebras generated by a finite set of partial isometries $\{ S_1, \dots , S_n\}$ subject to some conditions.
In this case there exists a partial representation of the free group $\F_n$ sending the $i$-th canonical generator to $S_i$. This idea was generalized in \cite{EL} to treat the case of infinite matrices and was used to give the first definition of Cuntz-Krieger algebras for transition matrices on infinitely many states.
\end{example}

Let $\pi : G \to \End_K(V)$  and $\pi' : G \to \End_K(W)$ be two partial representations of $G$. A {\it morphism of partial representations}  is a linear map $f : V \to W$  such that $f \circ \pi (g) = \pi' (g) \circ f$ for any $g \in G$.
The category of partial representations of $G$, denoted as $\ParRep G$ is the category whose objects are pairs $ (V,\pi)$, where $V$ is a $K$-vector space and $\pi : G \to \End_K(V)$ is a partial representation of $G$ on $V$, and whose morphisms are morphisms of partial representations.

Let $B$, $C$ be algebras and $\pi_{1}: G \to \End_{K} B$, $\pi_{2}: G \to \End_{K} C$ partial representations. Then we can define a partial representation of $C^{op}$, $\pi_{2}^{op}:G \to \End_K C^{op}$ given by $\pi_{2}^{op}(g) = \pi_{2}(g^{-1})$ and so a partial representation of $B { \otimes} C^{op}$ given by $\pi(g) =  \pi_{1}(g) \otimes \pi_{2}^{op}(g)$.

In order to study the representations of $G$ one can consider the group algebra ${\rm K} G$ which is an associative algebra with the same representation theory of the group $G$; in a similar fashion, we can define the partial group algebra $\Kpar G$, whose representations are in one-to-one correspondence with the partial representations of $G$ as follows.

\begin{definition}
Given a group $G$ and a field $K$, the {\it partial group algebra} $\Kpar G$ is the universal $K$-algebra with unit $1$ generated by the set of symbols $\{ [g] : g \in G \}$, with relations:

\begin{itemize}
\item [(1)] $[e] = 1$;
\item [(2)] $[s^{-1}][s][t] = [s^{-1}][st]$;
\item [(3)] $ [s][t][t^{-1}] = [st][t^{-1}]$;
for all $s, t \in G$.
\end{itemize}
\end{definition}

Clearly  the map $G \to \Kpar G$ given by $g \mapsto [ g ]$  is a partial representation of the group $G$ on the algebra $\Kpar G$.

\begin{teo}
The category $\ParRep G$ is equivalent to the category $\Rep \Kpar G$.
\end{teo}

\begin{proof}
It is straightforward to check that, if $V$ is any K-vector space and $\pi_V : G \to \End_K(V)$ is a partial representation of $G$ in $V$, then $\pi_V$ extends uniquely by linearity to a representation $\phi_V  : \Kpar G \to \End_K(V)$ such that $\phi_V([g]) = \pi_V(g)$, that is,
\[ \xymatrix{
G  \ar[drr]_{\pi_V}  \ar[rr]^{ g \mapsto [g]} & &\Kpar G  \ar[d]^{\phi_V} \\
& & \End_K(V) }\]
Conversely, if $\phi_V : \Kpar G \to \End_K(V)$ is a representation, then $\pi _V (g) = \phi_V ([g])$ gives a partial representation of $G$ in $V$. 
\end{proof}

To see how partial representations of groups are closely related to the concept of partial actions of groups, let us briefly remember some facts about partial group actions.

\begin{definition} Let $G$ be a group and $A$ an algebra, a {\it partial action} $\alpha$ of $G$ on $A$ is given by a collection $\{D_g\}_{ g \in G }$ of ideals of $A$ and a collection $\{\alpha_g : D_{g^{-1}} \to  D_g\}_{g\in G}$ of (not necessarily unital) algebra isomorphisms, satisfying the following conditions:
\begin{itemize}
\item [(1)] $D_e =A$, and $\alpha_e =\id_A$;
\item [(2)] $\alpha_h (D_{h^{-1}} \cap D_{(gh)^{-1}})=D_h \cap D_{g^{-1}}$;
\item [(3)] If $x \in D_{h^{-1}} \cap D_{(gh)^{-1}} $, then $\alpha_g \alpha_h (x) = \alpha_{gh} (x) $.
\end{itemize}

It can be easily seen that condition $(2)$ can be replaced by
the ÒweakerÓ condition: $\alpha_h ( D_{(gh)^{-1}})  \supseteq D_h \cap D_{g^{-1}}$.

\begin{example} \label{restriction}
An action of $G$ on an algebra $A$ is clearly a partial action, defining $D_g= A$ for all $g \in G$ and $\alpha_g$ the map $a \in A \mapsto g(a) \in A$. Moreover, 
every unital ideal of $A$ carries a partial action: if $B$ is such an ideal, with unit $1_B$, then 
a partial $G$-action $\beta$ on $B$ is obtained by defining $D_g = B \cap g(B)$ and $\beta_g$ to be the restriction of $\alpha_g$ to the ideal $D_{g^{-1}}$.  Note that each ideal $D_g$ of $B$  is also unital, with unit  $u_g = 1_B g(1_B)$.
\end{example}

Consider two partial actions $(A, \{D_g\}_{g\in G}, \{\alpha_g\}_{g \in G})$ and $(B, \{E_g\}_{g\in G}, \{\beta_g\}_{g \in G})$. A {\it morphism of partial actions} $$\varphi: (A, \{D_g\}_{g\in G}, \{\alpha_g\}_{g \in G}) \to (B, \{E_g\}_{g\in G}, \{\beta_g\}_{g \in G})$$  is an algebra morphism $\varphi : A \to B$ such that $\varphi (D_g) \subset E_g$ and 

\[ \xymatrix{
D_{g^{-1}} \ar[r]^{\alpha_g}  \ar[d]_{\varphi} &  D_g \ar[d]^{\varphi} \\
E_{g^{-1}} \ar[r]_{\beta_g}  &  E_g 
}\]
for all $g \in G$.
Partial actions and the morphisms between them form a category that we denote as $\ParAct G$.
\end{definition}

\begin{remark}
\begin{itemize}
\item [a)]
Since the domain of  $\alpha_g \alpha_h$ is $\alpha_h^{-1}(D_h \cap D_{g^{-1}})$, conditions $(2)$ and $(3)$ in the previous definition say that $ \alpha_{gh}$ is only an extension of $\alpha_g \alpha_h$.  However, the partial relations hold, that is, 
$\alpha_g \alpha_h \alpha_{g^{-1}} =   \alpha_{gh} \alpha_{g^{-1}}$ and $\alpha_{h^{-1}} \alpha_g \alpha_h  =  \alpha_{h^{-1}} \alpha_{gh}$.  In fact these partial relations can be used to rephrase the 
definition of partial action, see \cite[Proposition 4.1]{E}.
\item [b)] If $A = \sum_{g \in G} A_g$ is a $G$-graded algebra, by definition the product $A_g A_h$ is contained in $A_{gh}$, but in general they do not coincide.  
However, if $A_g A_{g^{-1}} A_g = A_g$ for any $g \in G$, the partial relations between ideals hold, that is, $A_g A_h A_{g^{-1}} =   A_{gh} A_{g^{-1}}$ and $A_{h^{-1}} A_g A_h  =  A_{h^{-1}} A_{gh}$, see \cite[Proposition 5.3]{E}.
\end{itemize}
\end{remark}

A partial action of a group $G$ on a algebra $A$ enables us to construct a new algebra, called the {\it partial smash product} 
(also referred to as the ``partial skew group ring''), denoted by $A \times_\alpha G$. Basically
$$A \times_\alpha G = \sum_{g \in G} D_g \# g $$
as a $K$-module and with the product defined as 
\[
(a_g  \# g)(b_h  \# h) =   \alpha_g( \alpha_{g^{-1}} (a_g)b_h) \# {gh}. 
\]
Note that $\alpha_{g^{-1}} (a_g) \in D_{g^{-1}}$, $b_h \in D_h$ and therefore 
\[
\alpha_g( \alpha_{g^{-1}} (a_g)b_h) \in \alpha_g(D_{g^{-1}} D_h) \subset \alpha_g(D_{g^{-1}} \cap D_h)
\subset D_g \cap D_{gh} \subset D_{gh}. 
\]

\begin{remark}
It is well known that the definitions of skew group rings and of smash products coincide when the Hopf algebra considered is $KG$.  Similarly, the definitions of partial skew group ring and of partial smash product coincide when the ideals $D_g$ are of the form $Au_g$.
\end{remark}

\begin{example}
Let $A$ be the commutative algebra $A=k[x,y]/ \langle x^2, y^2 \rangle$, $G= \langle g: g^2=1 \rangle$ the cyclic group of order $2$ and $I=Ay$ the ideal generated by $y$ (generated by $y$ and $xy$ as a vector space). Consider the partial action $\alpha$ of $G$ on $A$ given by $D_g =I$, $\alpha_g(y)= xy$, $\alpha_g(xy)= y$. 
Then the partial smash product $A \times_\alpha G$ is not associative. More precisely, taking $u = x \delta_1 + xy \delta_g$ we have that $(uu)u=0$ and $u(uu)= xy\delta_g$, see \cite[Example 3.5]{DE}.
\end{example}

From now on we assume that the domains $D_g$ are ideals of the form $Au_g$ where  
the generators $u_g$ are central idempotents of $A$ for each $g \in G$. 
This condition naturally appears, for instance, in the description of a Leavitt path algebra
as a partial smash product \cite{Goncalves} and in the development of the Galois theory for partial actions \cite{DFP07,BP12,KS14}; 
it also determines whether the partial action can be obtained as a restriction as in Example \ref{restriction} \cite[Theorem 4.5]{DE}.
In this case the partial smash product is automatically associative and the formula of the product in $A \times_\alpha G$ simplifies to 
$$(a u_g   \# g)(b u_h  \# h) =  a  \alpha_g( b  u_h u_{g^{-1}} ) u_{gh}  \# {gh}.$$

It is easy to verify that the map $\pi_0  : G \to A \times_\alpha G$, given by $\pi_0 (g) = u_g\# g$ is a partial representation of the group $G$ on the algebra $A \times_\alpha G$.

The partial smash product has an important universal property. Let $A$ be an algebra on which the group $G$ acts partially, consider the canonical inclusion
$\phi_0 : A \to  A \times_\alpha G$ defined by $\phi_0 (a) = a u_e \#e$
which is easily seen to be an algebra monomorphism. Given a $K$-vector space $V$, a pair of maps $(\phi_V,\pi_V) $ is said to be a {\it covariant pair} if $\phi_V  : A \to \End_K(V)$ is a representation and $\pi_V  : G \to \End_K(V)$  is a partial representation such that
$$ \phi_V(\alpha_g(au_{g^{-1}} )) = \pi_V (g) \phi_V (a) \pi_V (g^{-1}).$$
We denote $\CovPair (A,G)$ the category whose objects are covariant pairs $(\phi_V,\pi_V) $, and a {\it morphism between covariant pairs} $f: (\phi_V,\pi_V) \to (\phi_{W},\pi_{W}) $  is a linear map $f : V \to W$ such that $f \circ \pi_V (g) = \pi_{W} (g) \circ f$  and $f \circ \phi_V (g) = \phi_{W} (g) \circ f$ for any $g \in G$.

\medskip

The universal property of $A \times_\alpha G$ is given by the following result.

\begin{teo} Let $A$ be an algebra on which the group $G$ acts partially, $V$ a $K$-vector space and $(\phi_V, \pi_V)$ a covariant pair related to these data. Then there exists a unique algebra morphism $\Phi :A \times _\alpha  G \to \End_K(V)$ such that
\[ \xymatrix{
& A \times_{\alpha} G  \ar[dd]_\Phi \\
A  \ar[ru]^{\phi_0}  \ar[rd]_{\phi_V} & & G \ar[lu]_{\pi_0}  \ar[dl]^{\pi_V} \\
& \End_K(V)
} \]   
is commutative.
\end{teo}

\begin{proof}
It is clear that the map $\Phi: A \times_\alpha G \to \End_k(V)$ defined by $\Phi (a u_g \# g) =  \phi_V (a) \pi_V (g)  $ gives the desired result. 
\end{proof}

\begin{coro} Let $A$ be an algebra on which the group $G$ acts partially.  Then the category $\Rep A\times_\alpha G$ is equivalent to the category of covariant pairs $\CovPair (A,G)$.
\end{coro}

A very important result in the theory of partial representations of groups is that the partial group algebra $\Kpar G$ is always isomorphic to a partial smash product. First, it is important to note that the partial group algebra $\Kpar G$ has a natural $G$-grading. Indeed we can decompose, as a vector space, the whole partial group algebra as
$$\Kpar G  = \sum_{g \in G} B_g,$$
where each subspace $B_g$ is generated by elements of the form $[h_1][h_2]...[h_n]$ such that $g = h_1h_2 ...h_n$, and it is easy to see that the product in $\Kpar G$ makes $B_gB_h \subset B_{gh}$. Now, for each $g \in G$ define the element $e_g = [g][g^{-1}]  \in \Kpar G$. One can prove easily that these $e_g$ are idempotents for each $g \in  G$. These elements satisfy the following commutation relation:
$$[g]e_h = e_{gh}[g].$$
Indeed,
\begin{align*}
[g]e_h & =  [g][h][h^{-1}] = [gh][h^{-1}] \\
& =  [gh][(gh)^{-1}][gh][h^{-1}] = [gh][(gh)^{-1}][g] \\
& = e_{gh}[g].
\end{align*} 
From this, one can prove that all $e_g$ commute among themselves. Define the subalgebra $B = <e_g|g \in G > \subset \Kpar G$. This is a commutative algebra generated by central idempotents, and it is not difficult to prove that the subalgebra $B$ corresponds to the uniform subalgebra $B_e$ coming from the natural $G$ grading above presented. Then, we have the following two results.

\begin{teo}
Given a group $G$, there is a partial action of $G$ on the commutative algebra $B$ above defined, such that $\Kpar G = B \times_\beta G$.
\end{teo}

\begin{proof} In order to define a partial action of $G$ on $B$, we have to give the domains $D_g$ and the isomorphisms $\beta_g : D_{g^{-1}} \to  D_g$ for each $g \in G$. As the elements $e_g$ are central idempotents in $B$, define the ideals $D_g = e_gB$. Clearly, these ideals are unital algebras with unit $e_g$. Now, the partially defined isomorphisms between these ideals are
\[\beta_g(e_{g^{-1}}e_{h_1} \cdots e_{h_n}) = [g]e_{g^{-1}}e_{h_1} 
\cdots e_{h_n}[g^{-1}] = e_ge_{gh_1} \cdots e_{gh_n}.\]
It is easy to verify that these data indeed define a partial action of $G$ on $B$.
In order to prove the isomorphism, let us use both universal properties, of the partial smash product and of the partial group algebra. First, the map $\pi_0 : G \to B\times_\beta   G$ given by $\pi_0(g) = e_g \# g$ is a partial representation of the group $G$ on the partial smash product. Then, there is a unique algebra morphism $\hat \pi : \Kpar G \to  B \times_\beta G$, which factorizes this partial representation. This morphism can be written
explicitly as
\[
\hat \pi ([g_1]...[g_n])=e_{g_1}e_{g_1g_2} \cdots e_{g_1 \cdots g_n}  \# g_1 \cdots g_n.
\]
On the other hand, the canonical inclusion of $B$ into $\Kpar G$ and the canonical
partial representation form a covariant pair relative to the algebra $\Kpar G$ then there is a unique algebra morphism $\varphi : B \times_\beta G \to  \Kpar G$ explicitly given by
\[\varphi ( e_ge_{h_1} \cdots e_{h_n} \# g ) = e_ge_{h_1}  \cdots e_{h_n}[g].
\]
Easily, one can verify that the morphisms $\hat \pi$ and $\varphi$ are mutually inverses, completing
the proof.
\end{proof}

\begin{teo}
The $K$-vector space $B$ admits a partial representation $\pi : G \to \End_K (B)$ defined by $\pi(g)(x) = [g] x [g^{-1}]$ for any $g \in G, x \in B$.
\end{teo}

\begin{proof}
The map $\pi(g)(x) = [g] x [g^{-1}]$ defines a partial representation since 
$\pi (e) = \id_V$ because $[e]=1$;  $\pi (s) \pi (t) \pi (t^{-1}) = \pi (st) \pi (t^{-1})$ because
\[   [s]  [t]  [t^{-1}]  x   [t]   [t^{-1}][s^{-1}] =  [st]  [t^{-1}]  x   [t]   [(st)^{-1}]
\]
and analogously, $\pi (s^{-1})\pi (s) \pi (t) = \pi(s^{-1})\pi(st)$.
\end{proof}

\section {Partial group cohomology}

In this section we define the partial group cohomology as the right derived functor of the functor of partial invariants. As a first step we show that the functor of partial invariants is representable, that is, $(-)^{G_{\mathrm{par}}} \simeq \Hom_{\Kpar G} (B, -)$. Later we relate this cohomology with partial derivations and with the partial augmentation ideal.\\

If $G$ is a group  and $\phi_V : \Kpar G \to \End_K(V)$ is an object in $\Rep \Kpar G$, the set of partial $G$-invariants of $V$  is defined as
\[
V^{G_{\mathrm{par}}} = \{ v \in V :  \phi _V([g])(v)  =\phi _V(e_g) (v)  \ \mbox {for all $g \in G$} \}.
\]
It is clear that  $V^{G_{\mathrm{par}}}$  is a $K$-vector space and if $f: V \to W$ is a morphism in $\Rep \Kpar G$ 
 and $v \in  V^{G_{\mathrm{par}}}$, then 
\[
\phi_W ([g]) ( f (v) )= f( \phi_V ([g])(v))  = f( \phi_V (e_g) (v) )=  \phi_W(e_g)(f(v)),
\]
hence $f$ induces a linear map 
$f^{G_{\mathrm{par}}} : V^{G_{\mathrm{par}}} \to W^{G_{\mathrm{par}}}$.

\begin{prop}
$(-)^{G_{\mathrm{par}}} : \Rep \Kpar G \to \Rep K$ is a left exact functor.
\end{prop}

\begin{proof}
From the previous discussion, it is clear that $(-)^{G_{\mathrm{par}}}$ is a functor. 
To see that it is left exact it is enough to see that there is a natural isomorphism
$$(-)^{G_{\mathrm{par}}} \simeq \Hom_{\Kpar G} (B, -)$$
given by 
$ v \mapsto f_v$ with $f_v(1) = v$. Observe that $f$ is uniquely defined by the element $f(1)$ since
\[
e_{g_1} e_{g_2} \cdots e_{g_m} = [h_1] [h_2] \cdots [h_m] 1 [h_m^{-1}] \cdots [h_1^{-1}]
\]
where $h_1= g_1$ and $h_i = g_{i-1}^{-1} g_i$ for any $i=2, \dots, m$.
Finally observe that the fact that any $f \in \Hom_{\Kpar G} (B, V)$ is a morphism in $\Rep \Kpar G$ implies that
\begin{align*}
\phi_V ([g])( f(1)) & =  f( \phi_B([g])(1) ) =  f([g] 1 [g^{-1}]) \\
&= f([g]  [g^{-1}] [g]  [g^{-1}])\\
&= f( \phi_B([g]   [g^{-1}] )(1) )\\
&= f( \phi_B(e_g) (1)) = \phi_V(e_g)( f(1))
\end{align*} 
and hence $f(1) \in V^{G_{\mathrm{par}}}$. 
\end{proof}

\begin{definition}
If $G$ is a group and $M$ is an object in $\Rep \Kpar G$, then the partial group cohomology groups of $G$ with coefficients in $M$ are defined as
\[ \H^n_{\mathrm{par}} (G, M ) = \Ext^n_{\Kpar G} (B, M),
\]
that is, $ H^n (G, M )$ is the right derived functor of $(-)^{G_{\mathrm{par}}} \simeq \Hom_{\Kpar G} (B, -)$.
\end{definition} 
In order to compute a $\Kpar G$-projective resolution of $B$ we start with the following exact sequence in $\Rep \Kpar G$ given by

\[
0 \to IG \to  \Kpar G \stackrel{\epsilon}{\to} B \to 0
\]
where $IG = \Ker \epsilon$ is the {\it partial augmentation ideal} and $\epsilon (   [g_1] \cdots [g_n] ) = e_{g_1} e_{g_1 g_2} \cdots e_{g_1 g_2 \cdots g_n}$. 

\begin{lemma}\label{epsilon}
The morphism $\epsilon: \Kpar G \to B$ given by $\epsilon (   [g_1] \cdots [g_n] ) = e_{g_1} e_{g_1 g_2} \cdots e_{g_1 g_2 \cdots g_n}$ verifies the following properties:
\begin{itemize}
\item [(a)] $\epsilon ( x y) x = x \epsilon (y)$ for any $x,y \in \Kpar G$;
\item [(b)] $\epsilon (xy) = \epsilon (xy) \epsilon (x)$ for any $x,y \in \Kpar G$.
\end{itemize}
\end{lemma}

\begin{proof}
Take $x =  [g_1] \cdots [g_r], y =  [h_1] \cdots [h_s]$.  Recall that $B$ is commutative, $e_g$ is idempotent and $[g] e_h = e_g e_{gh} [g]$.  Then we have that 
\begin{align*}
x \epsilon (y) & =  [g_1] \cdots [g_r] e_{h_1} e_{h_1 h_2 } \cdots e_{h_1 h_2 \cdots h_s}  \\
& =  e_{g_1} e_{g_1 g_2} \cdots e_{g_1 \cdots g_r h_1 h_2 \cdots h_s}  [g_1] \cdots [g_r]  \\
& =  \epsilon ( x y) x 
\end{align*} 
and
\begin{align*}
\epsilon (xy) \epsilon (x) & =  e_{g_1} e_{g_1 g_2} \cdots e_{g_1 \cdots g_r h_1 h_2 \cdots h_s} e_{g_1} e_{g_1 g_2} \cdots e_{g_1 \cdots g_r} \\
& =  \epsilon ( x y).
\end{align*} 
\end{proof}

Now we define the vector space of partial derivations as follows: 

\begin{align*}
\Der_par (G, M)  =   \{ \delta \in \Hom_{B} (\Kpar G, M)  : \ &   \delta( a . b ) = a  \delta ( b ) +   \delta (a \epsilon (b) )   \\ 
&  \mbox{ for any $ a,b \in \Kpar G$}  \}.
\end{align*}
In particular, we say that $\delta \in \Der_par (G,M)$ is inner if $\delta ([g]) = [g] m - e_g m$ for some $m \in M$.  We denote by $\Int_par (G, M) $ the space of inner partial derivations.

\begin{prop}
There is a natural isomorphism 
\[
 \Hom_{\Kpar G} (IG, -) \simeq \Der_par (G, - ).
 \]
\end{prop}

\begin{proof}
The map  $$\Hom_{\Kpar G} (IG, M) \to  \Der_par (G, M) $$ given by
\[
f \mapsto \hat f, \ \mbox{with $ \hat f ( x) = f (  x-    \epsilon (  x ) . 1   )$}
\]
is a natural isomorphism of vector spaces.  It is clear that $\hat f \in \Hom_B (\Kpar G, M)$ since $f$ and $\epsilon$ are $B$-morphisms, and using Lemma \ref{epsilon} we get that $\hat f$ is a partial derivation:
\begin{align*}
\hat f (xy) & =   f (xy- \epsilon (xy) .1) =  f( xy - x \epsilon (y) +  x \epsilon (y)   - \epsilon (xy) ) \\
& =  x f( y -  \epsilon (y) )+  f ( x   \epsilon (y)  - \epsilon(xy)) \\
& =  x \hat f (y) +  \hat f (x \epsilon (y))
\end{align*}
because
\[ \epsilon(x \epsilon (y)) = \epsilon( \epsilon (xy) x) = \epsilon(xy) \epsilon (x) = \epsilon(xy).\] 
\end{proof}

\begin{teo}
Let $G$ be a group and $M$ an object in $\Kpar G$.  Then
\begin{align*}
\H^0_{\mathrm{par}} (G, M )  & =   M^{G_{\mathrm{par}}} = \Hom_{\Kpar G} (B, M);\\
\H^1_{\mathrm{par}} (G, M )  & = \Der_par (G, M) / \Int_par (G, M);\\
\H^n_{\mathrm{par}} (G, M )  & = \Ext^{n-1}_{\Kpar G} (IG, M), \  n \geq 2.
\end{align*}
\end{teo}

\begin{proof}
Associated to the short exact sequence 
\[
0 \to IG \to  \Kpar G \stackrel{\epsilon}{\to} B \to 0
\] there is a long exact sequence
{\small
\[ \xymatrix{
0 \ar[r] & \Hom_{\Kpar G} (B, M) \ar[r] & \Hom_{\Kpar G} (\Kpar G, M) \ar[r] & \Hom_{\Kpar G} (IG, M)    \\
\ar[r] & Ext^1_{\Kpar G} (B, M)\ar[r]  & Ext^1_{\Kpar G} (\Kpar G, M)\ar[r] & Ext^1_{\Kpar G} (IG, M)   \\
\ar[r] & Ext^2_{\Kpar G} (B, M)\ar[r]  & Ext^2_{\Kpar G} (\Kpar G, M)\ar[r] & \dots & }\] 
}Since $\Kpar G$ is projective, we have that $ \Ext_{\Kpar G}^n (\Kpar G, M) = 0$ for any $n \in \N$, so 
$$\H^n_{\mathrm{par}} (G, B) = \Ext^{n-1}_{\Kpar G} (IG, M)$$ for any $n \geq 2$.
Finally $\H^1_{\mathrm{par}} (G, B)$ is the cokernel of the map $$M \simeq \Hom_{\Kpar G} (\Kpar G, M)  \to \Hom_{\Kpar G} (IG, M)$$ and hence the commutative diagram
{\small \[\xymatrix{
\Hom_{\Kpar G} (\Kpar G, M) \ar[r] \ar[d]^{\cong} & \Hom_{\Kpar G} (IG, M) \ar@{->>}[r] \ar[d]^{\cong} & Ext^1_{\Kpar G} (B, M)  \ar@{-->}[d] \\
\Int_par(G,M) \ar[r] & \Der_par (G,M) \ar@{->>}[r] & \Der_par (G, M) / \Int_par (G, M) 
}\] }yields the desired result.
\end{proof}

\section{Spectral sequence}

In this section we will show that there exists a Grothendieck spectral sequence  relating cohomology of partial smash products with partial group cohomology and algebra cohomology:

\begin{teo}\label{espectral}
For any $A \times_\alpha G$-bimodule $M$ there is a third quadrant cohomology spectral sequence starting with $E_2$ and converging to $H^*(A \times_\alpha G, M)$:
$$ E_2^{p,q} = H_{\mathrm{par}}^q(G, H^p(A,M)) \Rightarrow H^{p+q} (A \times_\alpha G, M).$$
\end{teo}

\medskip
We start by studying the behavior of the functors that are considered in the mentioned spectral sequence. For any pair of objects 
$${ \phi_X: \Kpar G} \to \End_K(X) \ \in \ \Rep \Kpar G$$ and $$\Phi_M:  (A \times_\alpha G)^e \to \End_K(M) \ \in \ \Rep (A \times_\alpha G)^e,$$ we consider the object $$\Delta: (A \times_\alpha G)^e \to \End_K(X \otimes_B M) \ \in \ \Rep (A \times_\alpha G)^e$$ given by
$$\Delta(au_g \# g \otimes bu_h \#h)(x \otimes m)= \phi_X([g])(x) \otimes \Phi_M(au_g \# g \otimes bu_h \#h)(m)$$
which is well defined since
\begin{align}
\Delta(au_g \# g & \otimes bu_h \#h)(\phi_X(e_s)(x) \otimes m) \nonumber  \\
& =  \phi_X([g])(\phi_X(e_s)(x)) \otimes \Phi_M(au_g \# g \otimes bu_h \#h)(m)  \nonumber \\ 
&= \phi_X([g] e_s)(x) \otimes \Phi_M(au_g \# g \otimes bu_h \#h)(m)  \nonumber  \\
&= \phi_X(e_{gs} [g])(x) \otimes \Phi_M(au_g \# g \otimes bu_h \#h)(m)  \nonumber  \\
&= \phi_X(e_{gs}) \phi_X( [g])(x) \otimes \Phi_M(au_g \# g \otimes bu_h \#h)(m)  \\
&= \phi_X([g])(x) \otimes \Phi_M (u_{gs} \# e \otimes 1) \Phi_M(au_g \# g \otimes bu_h \#h)(m)  \\
&= \phi_X([g])(x) \otimes \Phi_M ((u_{gs} \# e)(au_g \# g)   \otimes bu_h \#h)(m)   \\
&= \phi_X([g])(x) \otimes \Phi_M ((au_g \# g)(u_{s} \# e)   \otimes bu_h \#h)(m)   \\
&= \phi_X([g])(x) \otimes  \Phi_M(au_g \# g \otimes bu_h \#h) \Phi_M (u_{s} \# e \otimes 1)(m) \nonumber  \\
&= \Delta(au_g \# g \otimes bu_h \#h)(x  \otimes \Phi_M (u_{s} \# e \otimes 1)(m)).  \nonumber 
\end{align} 
From  $(1)$ to $(2)$  we use that $B$ is a commutative ring. 
 From $(3)$ to $(4)$ we use the equality  $(u_{gs} \#  e)(a u_g \# g) = (a u_g \# g)(u_{s} \# e)$ which can be deduced as follows: $(u_{gs} \#  e)(a u_g \# g) = u_{gs} a u_{g} \# g$ and, on the other hand,  
 \begin{align*}
 (a u_g \# g)(u_{s} \# e) & = a u_g \alpha_{g}(u_{s} u_{g^{-1}}) u_{g} \# g \\
 & = a u_g \alpha_{g}(u_{s} u_{g^{-1}})  \# g.
 \end{align*}
 Now  $u_s u_{g^{-1}} \in D_{s} D_{g^{-1}} = D_s \cap D_{g^{-1}} \stackrel{\alpha_{g}}\rightarrow D_{gs} \cap D_{g} = D_{gs} D_{g}.$ So, $\alpha_{g}(u_{s} u_{g^{-1}}) = u_{gs}u_{g}$ and then $a u_g \alpha_{g}(u_{s} u_{g^{-1}})  \# g = a u_{g} u_{gs}  \# g$.  \\

In particular, if we take $M= A \times_\alpha G$ we have that $X \otimes_B (A \times_\alpha G)$ is an object in $ \Rep (A \times_\alpha G)^e$.

On the other hand, it is clear that $M$ can be viewed as an object in $\Rep A^e$, where $\phi_M: A^e \to \End_K(M)$ is the composition
$$ A^e  \stackrel{\phi_0 \otimes \phi_0}{\longrightarrow} (A \times_\alpha G)^e \stackrel{\Phi_M}{\longrightarrow} \End_K(M)$$
given by
$$ a \otimes b \mapsto a u_e   \# e  \otimes b u_e \# e  \mapsto \Phi_M ( a u_e  \# e \otimes b u_e  \# e)$$
and then we can consider the object
$$\pi : G \to \End_K ( \Hom_{A^e}(A,M)   ) \ \in \ \Rep \Kpar G$$
given by
$$\pi(g) (f)(x) = \Phi_M (u_g \# g \otimes u_{g^{-1}}\# g^{-1} ) f (\alpha_{g^{-1}}( u_g x)) . $$
It is clear that $\pi(g)(f) \in  \Hom_{A^e}(A,M)$ since
\begin{align}
&\pi(g) (f) (axb) \nonumber \\
 & =  \Phi_M (u_g \# g \otimes u_{g^{-1}}\# g^{-1} ) f( \alpha_{g^{-1}} (u_g axb))  \nonumber \\
& =    \Phi_M (u_g \# g  \otimes u_{g^{-1}}\# g^{-1} ) f( \alpha_{g^{-1}} (u_g a) ( \alpha_{g^{-1}} (u_g x) \alpha_{g^{-1}} (u_g b)) \nonumber\\
& =    \Phi_M (u_g \# g  \otimes u_{g^{-1}}\# g^{-1} ) \Phi_M (\alpha_{g^{-1}} (u_g a) u_e \# e  \otimes \alpha_{g^{-1}} (u_g b) u_e \# e )  \nonumber \\
& \ \ \ f( \alpha_{g^{-1}} (u_g x)) \nonumber \\
& =    \Phi_M ((u_g \# g) (\alpha_{g^{-1}} (u_g a) u_e \# e) \otimes (\alpha_{g^{-1}} (u_g b) u_e \# e)   (u_{g^{-1}} \# g^{-1})) \nonumber\\
& \ \ \ f( \alpha_{g^{-1}} (u_g x))  \\
& =    \Phi_M (   (a u_e  \# e) (u_g \# g)     \otimes     (u_{g^{-1}} \# g^{-1}) (b u_e \# e)   )   f( \alpha_{g^{-1}} (u_g x))   \\
& =    \Phi_M (   a u_e \# e     \otimes     b u_e \# e   )
\Phi_M (  u_g \# g    \otimes     u_{g^{-1}} \# g^{-1}  )
f( \alpha_{g^{-1}} (u_g x)) \nonumber\\
& =    \Phi_M (   a u_e \# e     \otimes     b u_e \# e   )  \pi(g) (f) (x) .  \nonumber
\end{align}
From $(5)$ to $(6)$ we use that
\begin{align*}
(u_g \# g) (\alpha_{g^{-1}} (u_g a) u_e \# e) & = u_g \alpha_{g}(\alpha_{g^{-1}} (u_{g} a) u_{g^{-1}}) u_{g} \# g  \\
& = u_g \alpha_{g}(\alpha_{g^{-1}} (u_{g} a) u_{g^{-1}})  \# g \\
& = u_g \alpha_{g}(\alpha_{g^{-1}} (u_{g} a) )  \# g \\
& = u_{g} a \# g = (au_e \# e) (u_g \# g)
\end{align*}
and analogously, 
\begin{align*}
(\alpha_{g^{-1}} (u_g b) u_e \# e)   (u_{g^{-1}} \# g^{-1}) & = \alpha_{g^{-1}}(u_{g}b) u_{g^{-1}}\# g^{-1} \\
& = u_{g^{-1}} \alpha_{g^{-1}}( b u_e u_g) u_{g^{-1}} \# g^{-1} \\
& = (u_{g^{-1}} \# g^{-1})(b u_e \# e).
\end{align*}  \\

This map $\pi$ induces a partial action since:
$$
\pi(e) (f) (x ) =  \Phi_M (u_e \# e  \otimes u_e \# e) f (\alpha_{e^{-1}}( x)) = f(x) $$
and 
$$\pi (g) \pi (h) \pi (h^{-1}) (f) (x)  = \pi (gh)  \pi (h^{-1}) (f) (x)$$
because
\begin{align}
& \pi (g) \pi (h) \pi (h^{-1}) (f) (x)  \nonumber \\
 & = \Phi_M (u_{g} \# g \otimes u_{g^{-1}}\# g^{-1} ) ( \pi (h) \pi(h^{-1})  f) (\alpha_{g^{-1}}( u_{g} x) ) \nonumber \\
& =  \Phi_M (u_{g} \# g \otimes u_{g^{-1}}\# g^{-1} )  \Phi_M (u_{h} \# h \otimes u_{h^{-1}}\# h^{-1} ) \Phi_M (u_{h^{-1}} \# h^{-1} \otimes u_{h}\# h ) \nonumber \\
& \ \ .  \ f (\alpha_{h}(u_{h^{-1}} (\alpha_{h^{-1}}( u_{h} (\alpha_{g^{-1}}( u_{g} x)))))) \\
& =  \Phi_M (u_{gh} u_{g}  \# g \otimes u_h u_{g^{-1}}\# g^{-1} ) f (\alpha_{h} (\alpha_{h^{-1}}( u_{h} (\alpha_{g^{-1}}( u_{g} x)))))    \\
& = \Phi_M (u_{gh} u_{g}  \# g \otimes u_h u_{g^{-1}}\# g^{-1} ) f (\alpha_{h} (\alpha_{h^{-1}}( u_{h} u_{g^{-1}} (\alpha_{g^{-1}}( u_{g} x)))))  \\
 & = \Phi_M (u_{gh} u_{g}  \# g \otimes u_h u_{g^{-1}}\# g^{-1} ) f (\alpha_{h} (\alpha_{h^{-1}}( \alpha_{g^{-1}}(u_{gh} u_{g}) (\alpha_{g^{-1}}( u_{g} x)))))  \\
& =  \Phi_M (u_{gh} u_{g}  \# g \otimes u_h u_{g^{-1}}\# g^{-1} ) f (\alpha_{h} (\alpha_{h^{-1}}( \alpha_{g^{-1}}(u_{gh} u_{g}  x))))  \\
& = \Phi_M (u_{gh} u_{g}  \# g \otimes u_h u_{g^{-1}}\# g^{-1} ) f (\alpha_{h} ( \alpha_{(gh)^{-1}}(u_{gh} u_{g}  x)))) \\
 & = \Phi_M (u_{gh} u_{g}  \# g \otimes u_h u_{g^{-1}}\# g^{-1} ) f (\alpha_{h} ( \alpha_{(gh)^{-1}}(u_{gh} u_{g} ) \alpha_{(gh)^{-1}} (u_{gh} x))) \nonumber \\
 & = \Phi_M ((u_{gh} \# gh)(u_{h^{-1}} \# h^{-1}) \otimes (u_h \# h)(u_{(gh)^{-1}}\# (gh)^{-1} ) \nonumber \\
 & \ \ .  \ f (\alpha_{h} ( u_{(gh)^{-1}} u_{h^{-1}} \alpha_{(gh)^{-1}} (u_{gh} x)) \nonumber \\
 & = \Phi_M ((u_{gh} \# gh) \otimes (u_{(gh)^{-1}}\# (gh)^{-1} ) \Phi_M ( (u_{h^{-1}} \# h^{-1}) \otimes (u_h \# h) ) \nonumber \\
 & \ \ .  \  f (\alpha_{h} (  u_{h^{-1}} \alpha_{(gh)^{-1}} (u_{gh} x)) \nonumber \\
& =  \pi(gh) \pi(h^{-1}) f(x). \nonumber
\end{align}
From $(7)$ to $(8)$  we use that $(u_g \# g) (u_h \# h) (u_{h^{-1} }\# h^{-1}) = (u_{gh} \# gh) (u_{h^{-1}} \# h^{-1})$, from $(9)$ to $(10)$ 
we use that $\alpha_{g^{-1}} (u_{gh} u_g) = u_h u_{g^{-1}}$ and from $(11)$ to $(12)$ we use that $\alpha_{(gh)^{-1}}(u_{gh} u_g) = \alpha_{h^{-1}}(\alpha_{g^{-1}}(u_{gh} u_{g}))$. \\

Now we consider the natural transformations
\[ \xymatrix{
\Hom_{\Kpar G} ( - , \Hom_{A^e} (A, M)) \ar@<-.5ex>[rr]_\Gamma &  & \ar@<-.5ex>[ll]_\Lambda \Hom_{(A \times_\alpha G)^e} ( - \otimes_B (A \times_\alpha G), M)
}\]
defined as follows: given $H \in \Hom_{\Kpar G} ( X , \Hom_{A^e} (A, M)) $, the map $\Gamma_X (H)$ is defined by
$$\Gamma_X (H)( x \otimes  au_g \# g )  :=
 \Phi_M (1 \otimes au_g \# g  ) H( x) (1) $$
and given $T \in  \Hom_{(A \times_\alpha G)^e} ( X \otimes_B (A \times_\alpha G) , M)$, the map $\ \Lambda_X (T)$ is defined by
$$\Lambda_X (T)(x) (a) : =  T( x \otimes au_e \#e ).$$
The map $\Gamma_X (H)$ is well defined since
\begin{align*}
  \Gamma_X (H) & ( e_h . x \otimes au_g \# g ) \\
 & =  \Phi_M (1 \otimes au_g \# g  ) H(e_h. x) (1)  \\
& =  \Phi_M (1 \otimes au_g \# g  ) (\phi(e_h) )H( x) (1)  \\
 & =  \Phi_M (1 \otimes au_g \# g  ) (\pi(h) \pi(h^{-1})H( x) (1)  \\
 & = \Phi_M (1 \otimes au_g \# g ) \Phi_M(u_h \#h \otimes  u_{h^{-1}} \#h^{-1}) \Phi_M(u_{h^{-1}} \#h^{-1} \otimes  u_{h} \#h) \\
 &  \ \ \ H (x)  (\alpha_h (u_{h^{-1}} \alpha_{h^{-1}}(u_h 1)))         \\
 & =  \Phi_M (1 \otimes au_g \# g ) \Phi_M(( u_h \#h) (u_{h^{-1}} \#h^{-1}) \otimes ( u_{h} \#h ) (u_{h^{-1}} \#h^{-1}))  \\
 & \ Ê\ \  H (x) (u_h)         \\
 & =  \Phi_M (1 \otimes au_g \# g ) \Phi_M(( u_h \#e)  \otimes ( u_{h} \#e ) )  H (x) (u_h)         \\
 & =  \Phi_M (u_h \# e \otimes ( u_h \#e) (au_g \# g) )  H (x) (u_h 1)         \\
& =  \Phi_M (1 \otimes (u_h \# e) (au_g \# g)  )  H  (x) ( u_h( u_h 1) )         \\
& =  \Phi_M (1 \otimes (u_h \# e) (u_h \# e) (au_g \# g)  )  H  (x) ( 1) )         \\
& =  \Phi_M (1 \otimes (u_h \#e)(au_g \# g) )  H  (x) ( 1 )         \\
& =   \Gamma_X (H)(x \otimes (u_h \#e)(a u_g \# g ) ) 
\end{align*} 
and
 \begin{align*}
  \Gamma_X (H) & (\Delta (cu_h\# h \otimes du_s \#s) (x \otimes au_g \# g )) =  \\
& =  \Gamma_X (H) (\pi_X(h) (x) \otimes  (cu_h\# h)(au_g \# g )(du_s \#s)  )  \\
& =   \Phi_M ( 1 \otimes (cu_h\# h)(au_g \# g )(du_s \#s)  ) H(\pi_X(h)( x)) (1)  \\
& = \Phi_M ( 1 \otimes (cu_h\# h)(au_g \# g )(du_s \#s)  ) \Phi_M(u_h \# h \otimes u_{h^{-1}} \# h^{-1} ) H(x)(1) \\
& =  \Phi_M ( u_h \# h  \otimes (\alpha_{h^{-1}}(cu_h)\# e)(au_g \# g )(du_s \#s )) H(x)(1)  \\
& =  \Phi_M ( u_h \# h  \otimes (au_g \# g )(du_s \#s )) H(x)(\alpha_{h^{-1}}(cu_h))  \\
& =  \Phi_M ( u_h \# h  \otimes (au_g \# g )(du_s \#s )) \Phi(\alpha_{h^{-1}}(cu_h)\# e \otimes (u_e \# e) ) H(x)(1)  \\
& =  \Phi_M ( ( u_h \# h)(\alpha_{h^{-1}}(cu_h)\# e)  \otimes (au_g \# g )(du_s \#s )) H(x)(1)\\
& =  \Phi_M ( ( u_h  \alpha_{h}(u_{h^{-1}} \alpha_{h^{-1}}(cu_h)) \# h)  \otimes (au_g \# g )(du_s \#s )) H(x)(1)  \\
& =   \Phi_M ( c u_h  \# h  \otimes (au_g \# g )(du_s \#s )) H(x)(1) \\
& =   \Phi_M ( c u_h  \# h  \otimes du_s \#s ) \Phi_M (1 \otimes au_g \# g ) H(x)(1) \\
& =   \Phi_M ( c u_h \# h  \otimes (du_s \#s )) \Gamma_X (H)  (x \otimes au_g \# g )).
\end{align*} 
On the other hand, $\Lambda_X (T) \in \Hom_{\Kpar G} ( X, \Hom_{A^e} (A, M))$ because 
\begin{align*}\Lambda_X (T)(x) (cad) & =   T( x \otimes (cu_e\#e)(au_e \#e) (du_e\#e))\\
 & =  \Phi_M(cu_e\#e \otimes du_e\#e)T(x  \otimes  au_e \#e )\\
 & =  \Phi_M(cu_e\#e \otimes du_e\#e) \Lambda_X (T)(x) (a)
\end{align*}
and
\begin{align*}
\Lambda_X (T)&(\pi_X(g) (x)) (a) \\
& =   T(  \pi_X(g)( x) \otimes au_e \#e) =
 T(  (\pi_X(g)\pi_X(g^{-1})\pi_X(g))( x) \otimes au_e \#e ) \\
&=  T(   (\pi_X(e_g)\pi_X(g))( x) \otimes au_e \#e) \\
&=  T(\pi_X(g)( x) \otimes (u_g \#e) (au_e \#e)   ) =  T( \pi_X(g)( x) \otimes au_g \#e   )\\
&=  T(\pi_X(g)( x) \otimes (u_g \#g)(\alpha_{g^{-1}}(au_g) \#e) (u_{g^{-1}} \#g^{-1}) )\\
&= \Phi_M (u_g\#g \otimes u_{g^{-1}}\#g^{-1}) T( x \otimes \alpha_{g^{-1}}(au_g) \#e ) \\
&=  \Phi_M (u_g\#g \otimes u_{g^{-1}}\#g^{-1}) \Lambda_X (T)(x) (\alpha_{g^{-1}}(au_g)) = \pi(g) \Lambda_X (T)(x) (a).
\end{align*}
Moreover, $\Lambda \circ \Gamma = \id$ because
$$\Lambda_X (\Gamma_X (H)) (x) (a) = \Gamma_X (H) ( x \otimes au_e \#e )= H (x) (a)$$
and $\Gamma \circ \Lambda = \id$ because
\begin{align*}
\Gamma_X (\Lambda_X (T) )(x \otimes au_g \# g ) &= \Phi_M (1 \otimes au_g \# g ) \Lambda_X (T)(x)(1) \\  
&= \Phi_M (1 \otimes au_g \# g ) T( x \otimes u_e \# e ) =
 T (x \otimes au_g \# g ).\end{align*}

The previous facts lead us to the following two propositions.

\begin{prop}
The functors $$\Hom_{\Kpar G} ( - , \Hom_{A^e} (A, M)) \ \mbox{and} \ \Hom_{(A \times_\alpha G)^e} ( - \otimes_B  (A \times_\alpha G), M)$$ are naturally isomorphic.
\end{prop}

\begin{proof}
The natural transformations $\Delta_X$ and $\Gamma_X$ yield the desired bijections.
\end{proof}

\begin{prop}\label{thm:4.2}
There exists a commutative  diagram of functors 
\[ \xymatrix{
 \Rep (A \times_\alpha G)^e \ar[drr]_{F_1} \ar[rrrr]^F & &  & & \Rep K  \\
& &  \Rep \Kpar G \ar[rru]_{F_2} } \]
where 
$$F(M) = \Hom_{(A \times_\alpha G)^e} ( A \times_\alpha G, M),$$
$$F_1(M) = \Hom_{A^e}(A,M)$$ and $$F_2(X) = \Hom_{\Kpar G}(B, X).$$
\end{prop}

\begin{proof}
We can apply the previous proposition in the particular case of $X=B$, and use the fact that $B \otimes_B (A \times_\alpha G) \simeq A \times_\alpha G$ as $A \times_\alpha G$-bimodules because
$$(au_g \# g) (1 \otimes x) (bu_h \# h) = e_g \otimes (au_g \# g) x (bu_h \# h)= 1 \otimes (au_g \# g) x (bu_h \# h).$$
\end{proof}

From \cite[Theorem 10.47]{R}, in order to finish the proof of Theorem \ref{espectral} we need the following proposition, whose proof will appear after some lemmas.

\begin{prop}\label{thm:4.3} 
The functor $F_2$ is left exact and $F_1 (M)$ is right $F_2$-acyclic for every
injective object $M$ in $\Rep (A \times_\alpha G)^e$.
\end{prop}

\begin{lemma} \label{principal}
 Let $S$ be a commutative semigroup where every element is an idempotent. Let $K$ be a field, let ${\rm K} S$ be the semigroup algebra of $S$. 
 If $I$ is a finitely generated ideal of ${\rm K} S$ then $I$ is principal and is generated by an idempotent of ${\rm K} S$.  
\end{lemma}

\begin{proof}
Let $I$ be a finitely generated ideal of ${\rm K} S$ and let  $r_1, \ldots, r_m$ be generators of this ideal.
Choose idempotents $u_1, \ldots, u_n$ of $S$ such that each $r_i$ is a combination of these idempotents, and let $T$ be the subsemigroup 
of $S$ generated by $u_1, \ldots, u_n$. $T$ is a commutative semigroup consisting only of idempotents, which 
is the same as a lower semilattice: the associated partial order is given by $u \leq v$ iff $uv = vu = u$,
and the greatest lower bound of $\{u,v\}$ is $uv$. Since $T$ is finite,  \cite[Theorem 1]{S} 
says that ${\rm K} T$ has a basis of orthogonal idempotents $w_1, \ldots, w_N$ (see also \cite[Theorem 4.2]{St}).

Each generator $r_i$ lies in ${\rm K} T$ and therefore we may write $r_i = \sum_{j} \alpha_{i,j} w_j$ for $i =1, \ldots, n$ (with $\alpha_{i,j}$
in $K$). 
Given that
$w_j r_i = \alpha_{i,j} w_j$, the set 
\[
 W = \{w_j ; \alpha_{i,j} \neq 0 \text{ for some } i\}
\]
is contained in $I$. On the other hand, every generator of $I$ is a $K$-linear 
combination of these elements and therefore the ideal generated by $W$ coincides with $I$. 

Finally, the ideal generated by $W$ is the ideal generated by the idempotent $u = \sum_{w_j \in W} w_j \in I$ which acts as an identity for the
elements of $I$. Since the $w_j$'s are mutually orthogonal,
$u w_j= w_j$ for each $j \in I$. Hence, if $y \in I$ then $y = \sum_{w_j \in W} b_jw_j$, with $b_j \in K S$, and therefore
$u y  =\sum_{w_j \in W} b_j (u w_j) = \sum_{w_j \in W} b_j w_j = y.$
\end{proof}

\begin{lemma} 
Every $B$-module $X$ is flat. 
\end{lemma}

\begin{proof}   
From \cite [Proposition 3.58]{R}, it is enough to show that for any finitely generated left ideal $I$ of $B$, 
 the morphism  $ I \otimes_B X \to B \otimes _B X \cong X$ is injective.
 By Lemma \ref{principal} and the fact that $B = {\rm K} S$, where $S$ is the commutative semigroup $S = \{e_{g_1}e_{g_2}\cdots e_{g_n}; g_i \in G, n \geq 1\}$,
 we have that each such ideal is principal and is generated by an idempotent $u$. 
 
Now assume that  $ \sum_{i} y_i \otimes x_i \in I \otimes_B X $ is such that
$ \sum_{i} y_i \cdot x_i=0$ in $X$. Since $y_i \in I$ for each $i$ we have $y_i = u y_i$ and therefore
\[
 \sum_{i} y_i \otimes_B  x_i = 
 \sum_{i} u y_i \otimes_B  x_i = 
  u  \otimes_B  (\sum_{i} y_i \cdot x_i) = 0,
\]
and it follows that  $ I \otimes_B X \to B \otimes _B X $ is injective.
\end{proof}

\begin{coro} \label{injective}
The functor $- \otimes_B (A \times_\alpha G):  \Rep \Kpar G   \to   \Rep (A \times_\alpha G)^e $ is exact.
\end{coro}

\medskip

\begin{proof}[Proof of Proposition \ref{thm:4.3}] It is clear that $F_2( - ) = \Hom_{\Kpar G}(B, - )$ is left exact.  If $M$ is an injective object in $\Rep (A \times_\alpha G)^e$, the isomorphism of functors 
$$\Hom_{\Kpar G} ( - , \Hom_{A^e} (A, M)) \simeq \Hom_{(A \times_\alpha G)^e} (- \otimes_B  (A \times_\alpha G), M)$$ and Corollary \ref{injective} 
imply that $\Hom_{\Kpar G} ( - , \Hom_{A^e} (A, M))$ is an exact functor. Hence $\Ext_{\Kpar G}^n(B, F_1(M))=0$ for any $n >0$  and so $F_1(M)$ is  $F_2$-acyclic.
\end{proof}


\begin{thebibliography}{99}

\bibitem{AAP}
 Abrams, Gene and Aranda Pino, Gonzalo,
{\it The {L}eavitt path algebra of a graph},
  J. Algebra {\bf 293} (2005), no. 2, 
 319--334.

\bibitem{BP12}
Bagio, Dirceu and Paques, Antonio,
{\it Partial groupoid actions: globalization, {M}orita theory, and
              {G}alois theory},
 Comm. Algebr  {\bf 40} (2012), no. 10, 3658--3678.
 
\bibitem{CK}
 {Cuntz, Joachim and Krieger, Wolfgang},
 {\it A class of {$C\sp{\ast} $}-algebras and topological {M}arkov
              chains},
    {Invent. Math.}  {\bf 56} (1980), no. 3,  {251--268}.
    
\bibitem{DE}
   {Dokuchaev, M. and Exel, R.},
     {\it Associativity of crossed products by partial actions,
              enveloping actions and partial representations},
{Trans. Amer. Math. Soc.}
   {\bf 357} (2005), no. 5, {1931--1952}.
   
   
\bibitem{DES08}
{Dokuchaev, M. and Exel, R. and Sim{\'o}n, J. J.},
      {\it Crossed products by twisted partial actions and graded
              algebras},
   {J. Algebra} {\bf 320} (2008), no. 8,  {3278--3310}.



\bibitem{DFP07}
Dokuchaev, Michael and Ferrero, Miguel and Paques, Antonio,
{\it Partial actions and {G}alois theory},
 {J. Pure Appl. Algebra} {\bf 208} (2007), no.1,  77--87.
 
 
    

\bibitem{E}
{Exel, Ruy},
 {\it Partial actions of groups and actions of inverse semigroups},
{Proc. Amer. Math. Soc.}  {\bf 126} (1998), no. 12, 
     {3481--3494}.



\bibitem{EL}
{Exel, Ruy and Laca, Marcelo},
{\it Cuntz-{K}rieger algebras for infinite matrices},
 {J. Reine Angew. Math.} {\bf 512} (1999), 
       {119--172}.



\bibitem{Goncalves}
 {Gon{\c{c}}alves, Daniel and {\"O}inert, Johan and Royer,
              Danilo},
  {\it Simplicity of partial skew group rings with applications to
              {L}eavitt path algebras and topological dynamics},
{J. Algebra} {\bf 420} (2014), {201--216}.



\bibitem {KS14} Kuo, Jung-Miao and Szeto, George,
{\it The structure of a partial {G}alois extension},
   Monatsh. Math.  { \bf 175} (2014), no. 4,  565--576.


\bibitem{KS16}
Kuo, Jung-Miao and Szeto, George,
{\it The structure of a partial {G}alois extension {II}},
   J. Algebra Appl. {\bf 15} (2016), no. 4, 1650061, 12pp.


 

	



 


\bibitem{QR}
  {Quigg, John and Raeburn, Iain},
     {\it Characterisations of crossed products by partial actions},
   {J. Operator Theory}  {\bf 37} (1997), no. 2, 
     {311--340}.

\bibitem{R}
 {Rotman, Joseph J.},
 { \it An introduction to homological algebra},
{Universitext},
 {Springer, New York},
   {2009},
    {xiv+709},
      ISBN = {978-0-387-24527-0}.
      
\bibitem{S}
    {Solomon, Louis},
  {\it The {B}urnside algebra of a finite group},
   {J. Combinatorial Theory} {\bf 2} (1967),  {603--615}.
	

\bibitem{St}
 {Steinberg, Benjamin},
 {\it M\"obius functions and semigroup representation theory},
    {J. Combin. Theory Ser. A}  {\bf 113} (2006), no. 5, 
     {866--881}.


      
	



\end{thebibliography}
\end{document}